\documentclass[11pt,a4paper]{amsart}

\usepackage{amsmath,amsthm,amssymb} 
\usepackage{cases}

\newtheorem{defi}{Definition}
\newtheorem{thm}[defi]{Theorem}
\newtheorem*{thm*}{Theorem}
\newtheorem{prop}[defi]{Proposition}
\newtheorem*{prop*}{Proposition}

\newtheorem{rem}[defi]{Remark}

\newtheorem*{fact*}{Fact}

\begin{document}

\title[Ruled minimal surfaces in pseudo-Euclidean space]
{On the classification of ruled minimal surfaces in pseudo-Euclidean space}
\author[Y. Sato]{Yuichiro Sato}
\address{Department of Mathematics and Information Sciences,
Tokyo Metropolitan University,
Minami-Osawa 1-1, Hachioji, Tokyo, 192-0397, Japan.}
\email{satou-yuuichirou@ed.tmu.ac.jp}


\date{\today}


\begin{abstract}
This paper gives, in generic situations, a complete classification of ruled minimal surfaces in pseudo-Euclidean space with arbitrary index. In addition, we discuss the condition for ruled minimal surfaces to exist, and give a counter-example on the problem of Bernstein type.
\end{abstract}

\maketitle
\section{INTRODUCTION}
In surface theory, research of ruled surfaces has a long history. In particular, there are many results on ruled minimal surfaces. For example, E. Catalan \cite{Ca2} proved that the non-planar, ruled minimal surface is a helicoid only in three dimensional Euclidean space. 
Recently, submanifolds of pseudo-Riemannian manifolds, for instance Lorentzian manifolds, are studied actively. They are focused not only in geometry but also in physics. 
Pseudo-Euclidean space is one of pseudo-Riemannian manifolds, and includes non-trivial ruled minimal surfaces besides helicoids. O. Kobayashi \cite{OK} classified spacelike ruled minimal surfaces in three dimensional Minkowski space, and I. van de Woestijne \cite{Wo} classified timelike ruled minimal surfaces in three dimensional Minkowski space. 
Thus, the classification for three dimensional Minkowski space was completed. More generally, H. Anciaux \cite{An} studied on ruled minimal surfaces in pseudo-Euclidean space with arbitrary index. He claimed that there is no new ruled minimal surface except for ones stated above. However, his proof is incomplete. 
We are motivated by his consequence and start to research. We constructed new examples of ruled minimal surfaces in four dimensional Minkowski space (see also \cite{KY1}). These examples are not isometric to any ruled minimal surface which has been obtained already. 

In this paper, we are inspired by Anciaux's proof, and give a complete classification of ruled minimal surfaces in $n$-dimensional pseudo-Euclidean space with arbitrary index $p$ (Theorem \ref{thm:2}). Moreover, we give the condition of ambient spaces for classified ruled minimal surfaces to exist in the space (Remark \ref{rem10}). 
Summing up these facts, we see that there are very fruitful ruled minimal surfaces in four dimensional Minkowski space or four dimensional pseudo-Euclidean space with neutral metric having index 2. In particular, it should be remarkable that some of those ruled minimal surfaces are embedded in three dimensional subspace with degenerate metric of pseudo-Euclidean space (Remark \ref{rem11}). This is one of the interesting results. We expect and hope that the study of spaces with degenerate metrics becomes important more and more. Moreover, in pseudo-Euclidean space whose dimension is greater than or equal to four, we remark that the problem of Bernstein type does not hold in a sense. 
%At the end of the introduction, we notice the following fact. Even if it seems that the structure of the space is very easy such as pseudo-Euclidean space, we should pay attention to research the classification of submanifolds sufficiently. 

In the section two, we give fundamental definitions and notations and state the necessary proposition to classify. In the section three, we classify ruled minimal surfaces in pseudo-Euclidean space. To classify, we consider cases of the cylinder type and non-cylinder type respectively. Here we discuss the existence of these surfaces and the problem of Bernstein type.
\section{PRELIMINARIES}
Let $I \subset \mathbb{R}$ be an open interval including 0 $\in \mathbb{R}$. 
Assume that $\gamma : I \rightarrow \mathbb{R}^n \setminus \{0\}$ is a $C^{\infty}$-curve and $x : I \rightarrow \mathbb{R}^n$ is a $C^{\infty}$-regular curve. 
Then, we define a mapping $f$ by the following 
\[f : I \times \mathbb{R} \ni (s,t) \longrightarrow \gamma(s)t + x(s) \in  \mathbb{R}^n.\]
From now on, we assume that $f$ is an immersion. The image $S:=  \{ \gamma(s)t + x(s) \in  \mathbb{R}^n \mid (s,t) \in I \times \mathbb{R} \}$ of this mapping $f$
is called a \textit{ruled surface} in $\mathbb{R}^n$. 
Moreover we define the curve $\gamma$ as a \textit{direction curve} on $S$, and the curve $x$ as a \textit{base curve} on $S$. 
In particular, if the direction curve is parallel, i.e. $\gamma(s) \equiv \gamma_{0} : \textrm{constant}$, 
then we say that a given ruled surface is \textit{cylinder}. If not so, we say that it is \textit{non-cylinder}.

As the ambient space, we consider pseudo-Euclidean space 
\[ \mathbb{R}_{p}^n := \left (\mathbb{R}^n, \langle \cdotp,\cdotp \rangle_{p} = -\sum_{i=1}^{p}dx_{i}^2+\sum_{j=p+1}^{n}dx_{j}^2 \right ), \]
where $n \geq 3$ and $0 \leq p \leq q := [\frac{n}{2}]$.
For each vector $v \in \mathbb{R}^{n}_{p}$, we should remark that the number $\langle v, v \rangle_{p}$ is not necessarily non-negative.
We define \textit{causal characters} for a non-zero vector $v \in \mathbb{R}^{n}_{p}$.
\begin{itemize}
\item A vector $v$ is called \textit{spacelike} if $\langle v, v \rangle_{p} > 0$.
\item A vector $v$ is called \textit{timelike} if $\langle v, v \rangle_{p} < 0$.
\item A vector $v$ is called \textit{null} or \textit{lightlike} if $\langle v, v \rangle_{p} = 0$.
\end{itemize}
A $C^{\infty}$-regular curve $c$ in $\mathbb{R}_{p}^n$ is called a \textit{null curve} if for any $s \in I$, it holds the condition $\langle c'(s), c'(s) \rangle_{p}=0$. 
Finally, a ruled surface $S$ in $\mathbb{R}_{p}^n$  is \textit{minimal} if the induced metric $g$ on $S$ is non-degenerate, and the mean curvature vector field $\vec{H}$ of $S$ is identically vanishing,
\[\textrm{i.e.} \ \det g=g_{11}g_{22}-g_{12}^2 \neq 0 , \quad \vec{H}=\frac{1}{2}\frac{g_{11}h_{22}-2g_{12}h_{12}+g_{22}h_{11}}{\det g} \equiv 0, \]
where $g_{ij}$ are the coefficients of the induced metric $g$ and $h_{ij}$ are the coefficients of the second fundamental form $h$ with respect to the coordinate system $(u_{1}, u_{2}) = (s, t)$. A non-degenerate submanifold in $\mathbb{R}_{p}^n$ is \textit{totally geodesic} if its second fundamental form vanishes identically. The classification of totally geodesic submanifolds in pseudo-Euclidean space was completed.
\begin{prop} \cite[Proposition 4, p.~13]{An} \label{totgeod} \rm
The totally geodesic submanifolds of the pseudo-Euclidean space $\mathbb{R}^{n}_{p}$ equipped with the above metric $\langle \cdotp,\cdotp \rangle_{p}$ are the open subsets of its non-degenerate affine subspaces.
\end{prop}
From now on, we use the following notations:\\
O.S. : orthogonal system, O.N.S. : orthonormal system, and O.N.F. : orthonormal frame.
\section{CLASSIFICATION}
In this section, we classify ruled minimal surfaces in $\mathbb{R}^n_{p}$.
\begin{thm} \cite[Theorem 8, p.~45]{An} \label{thm:1} Let $S$ be a cylinder ruled surface. 
If $S$ is minimal, then it is an open subset of one of the following surfaces:
\begin{enumerate}
\item[1.] a plane,
\item[2.] a minimal cylinder,
\end{enumerate}
where minimal cylinders satisfy that $\gamma_{0}$ is a null vector and $x(s)$ is a null curve such that $\langle \gamma_{0}, x'(s)\rangle_{p} \neq 0$.
\end{thm}
Next, we classify non-cylinder ruled minimal surfaces. 
\begin{prop} \cite[Lemma 5, p.~41]{An} \label{prop:1} Suppose that a direction curve $\gamma$ satisfies that $\langle \gamma(s), \gamma(s) \rangle_{p} \equiv 0$ and $\gamma$ is not parallel. Then, the ruled surface $S$ is not minimal.
\end{prop}
In this paper, we classify non-cylinder ruled minimal surfaces in generic situations. Here, a ruled surface is said to be \textit{singular} at a point $s_{0} \in I$ if it satisfies that,
\begin{itemize}
\item[(i)] for any of functions $\langle \gamma, \gamma \rangle_{p}, \langle \gamma', \gamma' \rangle_{p}, \langle x', x' \rangle_{p}$ and $\langle \gamma', x' \rangle_{p}$, it is zero at $s_{0}$ and it is not a zero-function on an arbitrary neighborhood at $s_{0}$, 
\end{itemize}
or
\begin{itemize} 
\item[(ii)] vectors $\gamma', x'$ are linearly dependent at $s_{0}$ and they are linearly independent at a point on an arbitrary neighborhood at $s_{0}$.
\end{itemize}
A ruled surface is in \textit{generic} situations if it is not singular at any point $s \in I$. 
Hereinafter, we make some preparation to classify. 
We can normalize $\langle \gamma(s), \gamma(s) \rangle_{p}=\epsilon=\pm1$ by Proposition \ref{prop:1} to be in generic situations. 

When the direction curve $\gamma(s)$ is not a null curve, let $s$ be the arc length parameter for $\gamma$. Then, we define
\[ \eta := \langle \gamma', \gamma' \rangle_{p} = \pm 1, 0. \]
When $\gamma(s)$ is a null curve and the base curve $x(s)$ is not a null curve, let $s$ be the arc length parameter for $x$. Then, we define
\[ \delta := \langle x', x' \rangle_{p} = \pm 1, 0. \] 
Moreover, for non-cylinder ruled minimal surfaces, we may assume $g_{12}=0$. 
In fact, set $\tilde{x}(s):=x(s)+\lambda(s)\gamma(s) \ (\lambda(s): \textrm{a real function})$, consider a mapping $f(s,t)=\gamma(s)t+\tilde{x}(s)$. Then, since we compute
\[f_{s}(s,t)=\gamma'(s)t+\tilde{x}'(s)=\gamma'(s)t+x'(s)+\lambda'(s)\gamma(s)+\lambda(s)\gamma'(s),\]
\[f_{t}(s,t)=\gamma(s),\]
note that $\langle \gamma(s),\gamma'(s)\rangle_{p}=\frac{1}{2}\frac{d}{ds} \langle \gamma(s), \gamma(s) \rangle_{p}=0$, 
we compute $g_{12}=\langle f_{s},f_{t}\rangle_{p}=\langle \gamma(s),x'(s)\rangle_{p}+\epsilon\lambda'(s)$. 
Again, we define $\lambda(s)$ as
\[\lambda(s)=-\epsilon\int_{0}^{s}\langle \gamma(u),x'(u)\rangle_{p}du.\]
We can take the base curve $x$ of the immersion $f$ which satisfies $g_{12}=0$ by this. 
This assumption is compatible with one that the curve $\gamma(s)$, or $x(s)$ is reparametrized by the arc length parameter. 
And, since $f_{ss}=\gamma''(s)t+x''(s), \  f_{st}=\gamma'(s), \ f_{tt}=0$, by the mean curvature formula, we have 
\begin{eqnarray}
2\vec{H}(s,t)=\frac{g_{22}h_{11}}{g_{11}g_{22}}=\frac{h_{11}}{g_{11}}. \label{eq:1}
\end{eqnarray}
By the formula (\ref{eq:1}), $S$ is minimal if and only if $h_{11}(s,t)=0 \ ((s,t) \in I  \times \mathbb{R})$. 
In addition, from $g_{11}(s,t)=\eta t^2+2\langle \gamma',x'\rangle_{p}t+\delta$, we get the following cases:
\begin{table}[htbp]
\begin{tabular}{|c|c|c|c|}
\hline
$\eta=\langle \gamma', \gamma' \rangle_{p}$ & $\delta=\langle x', x' \rangle_{p}$ & $\langle \gamma', x'\rangle_{p}$ & $\empty$ \\ \hline
$\pm1$ & non-zero & 0 & (i) \\ \hline
$\pm1$ & 0 & whatever & (ii) \\ \hline
$\pm1$ & non-zero & non-zero & (iii) \\ \hline
0 & $\pm1$ & non-zero & (iv) \\ \hline
0 & $\pm1$ & 0 & (v) \\ \hline
0 & 0 & non-zero & (vi) \\ \hline
0 & 0 & 0 & (vii) \\ \hline
\end{tabular}
\caption{} \label{tab:1}
\end{table}

In the case (vii), the conditions imply that the induced metric $g \equiv 0$. Thus, the case (vii) is excluded.
Now, $h_{11}(s,t) \equiv 0$ implies $f_{ss}(s,t) \in \textrm{span}\{f_{s},f_{t}\}$. 
Since $g_{11} \neq 0, \ g_{12}=0, \ g_{22}=\epsilon \neq 0$, we see that
\[ \left\{\frac{f_{s}}{\sqrt{ \left| \langle f_{s}, f_{s} \rangle_{p} \right| }}, \frac{f_{t}}{\sqrt{ \left| \langle f_{f}, f_{t} \rangle_{p} \right| }}\right\} \]
is an O.N.F. on $S$. Thus, we get
\[ f_{ss}=\frac{\langle f_{ss},f_{s}\rangle_{p}}{\langle f_{s}, f_{s} \rangle_{p}}f_{s}+\frac{\langle f_{ss},f_{t}\rangle_{p}}{\langle f_{t}, f_{t} \rangle_{p}}f_{t}.\]
And, noting facts that $\langle \gamma',\gamma''\rangle_{p}=\frac{1}{2}\frac{d}{ds}\langle \gamma'(s), \gamma'(s) \rangle_{p}=0, \ \langle \gamma,\gamma'\rangle_{p}=0$ and  
$\frac{d}{ds}\langle \gamma,\gamma'\rangle_{p}=\langle \gamma', \gamma' \rangle_{p}+\langle \gamma,\gamma''\rangle_{p}=0$ lead to $\langle \gamma,\gamma''\rangle_{p}=-\eta$, we calculate
\begin{equation}
f_{ss}=C(s, t)(\gamma't+x')+\epsilon(-\eta t+\langle \gamma,x''\rangle_{p})\gamma, 
\label{eq:2}
\end{equation}
where $C(s, t)$ is defined as
\[ C(s, t) := \frac{(\langle \gamma'',x'\rangle_{p}+\langle \gamma',x''\rangle_{p})t+\langle x'',x'\rangle_{p}}{\eta t^2+2\langle \gamma',x'\rangle_{p}t+\delta}. \]
On the other hand, we calculate directly
\begin{eqnarray}
f_{ss}=\gamma''(s)t+x''(s). \label{eq:3}
\end{eqnarray}
\begin{prop}\label{prop:2} 
Let $S$ be a non-cylinder ruled minimal surface. If $\gamma'(s),x'(s)$ are linearly independent for any $s \in I$, then $C(s,t)$ does not depend on the variable $t$, moreover $C(s,t) \equiv 0$ holds when $\eta=\pm 1$.
\end{prop}
\begin{proof}
From (\ref{eq:2}) and (\ref{eq:3}), it holds
\begin{eqnarray}
tC(s,t)\gamma'+C(s,t)x'-\epsilon\eta t\gamma+\epsilon\langle \gamma,x''\rangle_{p}\gamma=\gamma''t+x''.  \label{eq:15}
\end{eqnarray}
Differentiating on both sides with respect to $t$,
\[ \left (C(s,t)+t\frac{\partial C}{\partial t}(s,t) \right )\gamma'+\frac{\partial C}{\partial t}(s,t)x'-\epsilon\eta\gamma=\gamma''.\]
Again, differentiating on both sides with respect to $t$, we compute 
\[ \left (t\frac{\partial^2 C}{\partial t^2}(s,t)+2\frac{\partial C}{\partial t}(s,t) \right )\gamma'+\frac{\partial^2 C}{\partial t^2}(s,t)x'=0.\]
Since $\gamma'(s),x'(s)$ are linearly independent, we have 
\[t\frac{\partial^2 C}{\partial t^2}(s,t)+2\frac{\partial C}{\partial t}(s,t)=0, \quad \frac{\partial^2 C}{\partial t^2}(s,t)=0.\]
Therefore, we see $\frac{\partial C}{\partial t}(s,t)=0$, i.e. $C(s,t)$ does not depend on $t$. 
When we compare the coefficient of degree one for $t$ and the constant coefficient on both sides of (\ref{eq:15}), we get
\begin{numcases}
  {}
  \gamma''(s) = C(s)\gamma'(s)-\epsilon\eta\gamma(s), & \label{eq:4} \\
  x''(s)=C(s)x'(s)+\epsilon\langle \gamma(s),x''(s)\rangle_{p}\gamma(s), & \label{eq:5}
\end{numcases}
where we simply express $C(s)$ because $C(s,t)$ does not depend on $t$. 
If $\eta=\pm 1$ holds, by the inner product of $\gamma'(s)$ for the formula (\ref{eq:4}), then it follows that $C(s) \equiv 0$.
\end{proof} 
When $\eta=0$, we prove the following proposition.
\begin{prop}\label{prop:3} 
Let $S$ be a non-cylinder ruled minimal surface. If $\eta=0$ and $\delta=\pm 1$ hold, then we have $C(s,t) \equiv 0$.
\end{prop}
\begin{proof}
From (\ref{eq:2}) and (\ref{eq:3}) in the case $\eta=0$, it holds that 
\[\frac{(\langle \gamma'',x'\rangle_{p}+\langle \gamma',x''\rangle_{p})t+\langle x'',x'\rangle_{p}}{2\langle \gamma',x'\rangle_{p}t+\delta}(\gamma't+x')+\epsilon\langle \gamma,x''\rangle_{p}\gamma=\gamma''t+x''.\]
Since $\langle \gamma'',x'\rangle_{p}+\langle \gamma',x''\rangle_{p}=\frac{d}{ds}\langle \gamma',x'\rangle_{p}, \ \langle x'',x'\rangle_{p}=\frac{1}{2}\frac{d}{ds}\langle x', x' \rangle_{p}=0$, we get 
\[ \left( \frac{d}{ds}\langle \gamma',x'\rangle_{p} \right)t(\gamma't+x')+\epsilon\langle \gamma,x''\rangle_{p}(2\langle \gamma',x'\rangle_{p}t+\delta)\gamma=(2\langle \gamma',x'\rangle_{p}t+\delta)(\gamma''t+x'').\]
When we compare the coefficient of degree one on both sides, it holds that 
\begin{eqnarray}
2\epsilon\langle \gamma,x''\rangle_{p} \langle \gamma',x'\rangle_{p}\gamma+\left( \frac{d}{ds}\langle \gamma',x'\rangle_{p} \right)x'=\delta\gamma''+2\langle \gamma',x'\rangle_{p}x''.  \label{eq:a}
\end{eqnarray}
Since $g_{12}=\langle \gamma,x'\rangle_{p}=0$, by the inner product of $x'$ for the formula (\ref{eq:a}), it follows that 
\[\frac{d}{ds}\langle \gamma',x'\rangle_{p}\delta=\delta\langle \gamma'',x'\rangle_{p},\]
hence
\[\langle \gamma'',x'\rangle_{p}+\langle \gamma',x''\rangle_{p}=\langle \gamma'',x'\rangle_{p}.\]
This implies that $\langle \gamma',x''\rangle_{p}=0$. By using this formula, if we consider the inner product of $\gamma'$ for the formula (\ref{eq:a}), we obtain a differential equation 
\[ \left(\frac{d}{ds}\langle \gamma',x'\rangle_{p}\right)  \langle \gamma',x'\rangle_{p}=0.\]
Thus, since we see $\langle \gamma',x'\rangle_{p}=\textrm{const}$, the numerator of $C(s,t)$ is $0$, i.e. $C(s,t) \equiv 0$.
\end{proof}
Here, we go back to Table \ref{tab:1} and determine ruled minimal surfaces in the cases (i)-(vi).
\subsection*{Cases of (i), (ii)} 
It is easy to prove that $\gamma', x'$ are linearly independent. Actually, we set 
\[\alpha\gamma'+\beta x'=0 \ \ (\alpha,\beta \in \mathbb{R}).\]
If we consider the inner product of $\gamma'$ in the case (i), then we have $\alpha=0$. And, $x' \neq 0$ leads to $\beta=0$. 
In the case (ii), we consider the inner product of $\gamma'$ if $\langle \gamma',x'\rangle_{p}=0$ or  the inner product of $x'$ if $\langle \gamma',x'\rangle_{p} \neq0$. Then, we get $\alpha=0$, and we also get $\beta=0$ because of regularity of the curve $x$, i.e. the cases (i) and (ii) yield the linear independence. 
So, since we can apply Proposition \ref{prop:2}, we get the formulas (\ref{eq:4}), (\ref{eq:5}) in these cases. The condition $\eta=\pm1$ implies that
\begin{numcases}
  {}
  \gamma''(s) = -\epsilon\eta\gamma(s), & \label{eq:6} \\
  x''(s)=\epsilon\langle \gamma(s),x''(s)\rangle_{p}\gamma(s). & \label{eq:7}
\end{numcases}
We can determine the direction curve $\gamma(s)$ and the base curve $x(s)$ by the formula (\ref{eq:6}) and (\ref{eq:7}). We get the following solutions by referring to \cite[pp.~44--45]{An}
\begin{eqnarray*}
f(s,t)&=&(\cos{s}e_{1}+\sin{s}e_{2})t+sv+x_{0} \quad (\textrm{if} \ \epsilon\eta=1); \\
f(s,t)&=&(\cosh{s}e_{1}+\sinh{s}e_{2})t+sv+x_{0} \quad (\textrm{if} \ \epsilon\eta=-1);
\end{eqnarray*}
where $e_{1}=\gamma(0), e_{2}=\gamma'(0), v=x'(0), x_{0}=x(0) \in \mathbb{R}^{n}_{p}$, and $e_{i}$ and $v$ are orthogonal to each other. Moreover, it holds that $\langle e_{1}, e_{1} \rangle_{p}=\langle e_{2}, e_{2} \rangle_{p}=\pm1$ when $\epsilon\eta=1$, and $\langle e_{1}, e_{1} \rangle_{p}=-\langle e_{2}, e_{2} \rangle_{p}=\pm1$ when $\epsilon\eta=-1$.

It is obvious that $v=x'(0)$ is non-null in the case (i). So, we express that
\[\exists C_{0} \neq 0:\textrm{const} \ \textrm{s.t.} \ v=C_{0}e_{3},\]
where $e_{3} \in \mathbb{R}^{n}_{p}$ is a unit vector, i.e. $\langle e_{3}, e_{3} \rangle_{p} = \pm 1$.
Finally, making a suitable translation and scaling, we may set $C_{0}=1, x_{0}=0$. Thus, we obtain the following two patterns:
\begin{align*}
f(s,&t)=(\cos{s}e_{1}+\sin{s}e_{2})t+se_{3}, \\
& \textrm{where} \ \{e_{1},e_{2},e_{3}\} : \textrm{O.N.S.}, \ \langle e_{1}, e_{1} \rangle_{p}=\langle e_{2}, e_{2} \rangle_{p}=\pm1, \ \langle e_{3}, e_{3} \rangle_{p}=\pm1; \\
f(s,&t)=(\cosh{s}e_{1}+\sinh{s}e_{2})t+se_{3}, \\
& \textrm{where} \ \{e_{1},e_{2},e_{3}\} : \textrm{O.N.S.}, \ \langle e_{1}, e_{1} \rangle_{p}=-\langle e_{2}, e_{2} \rangle_{p}=\pm1, \ \langle e_{3}, e_{3} \rangle_{p}=\pm1;
\end{align*}
where the the double signs are arbitrary. 
On the other hand, since we see $v$ is null in the case (ii), again we put $e_{3}:=v$ and we obtain the following two patterns:
\begin{align*}
f(s,&t)=(\cos{s}e_{1}+\sin{s}e_{2})t+se_{3}, \\
& \textrm{where} \ \{e_{1},e_{2},e_{3}\} : \textrm{O.S.}, \ \langle e_{1}, e_{1} \rangle_{p}=\langle e_{2}, e_{2} \rangle_{p}=\pm1, \ \langle e_{3}, e_{3} \rangle_{p}=0; \\
f(s,&t)=(\cosh{s}e_{1}+\sinh{s}e_{2})t+se_{3}, \\
& \textrm{where} \ \{e_{1},e_{2},e_{3}\} : \textrm{O.S.}, \ \langle e_{1}, e_{1} \rangle_{p}=-\langle e_{2}, e_{2} \rangle_{p}=\pm1, \ \langle e_{3}, e_{3} \rangle_{p}=0.
\end{align*}
\subsection*{Case of (iii)}
First, if we assume that $\gamma', x'$ are linearly independent, since we can apply Proposition \ref{prop:2}, we don't obtain new results. In fact, since we get $C(s)=0, \langle x', x' \rangle_{p} \neq 0$, they imply the reduction to the case of (i). Therefore, we assume that $\gamma', x'$ are linearly dependent on $I$, i.e. there exist a function $p(s)$ on the interval $I$ such that
\begin{eqnarray*}
x'(s)=p(s)\gamma'(s).
\end{eqnarray*}
Then we obtain
\[ f_{s}(s,t) = \gamma'(s)t + x'(s) = (p(s) + t)\gamma'(s). \]
We remark that $p(s) + t \neq 0$ for each $(s,t) \in I \times \mathbb{R}$ since the non-degeneracy implies $g_{11} = \langle f_{s}, f_{s} \rangle_{p} = (p(s) + t)^{2}\eta \neq 0$.
However, the minimality implies that the image of the immersion is a plane. 
Actually, the formula (\ref{eq:1}) gives $h_{11}=0$. And, we recall $h_{22}=0$ for a general ruled surface $f(s,t)=\gamma(s)t+x(s)$. Therefore, it suffices to prove $h_{12}=0$.
We again compute that $f_{st}(s,t)=\gamma'(s)$. By using $p(s) + t \neq 0$, we see
\[ f_{st}(s,t) = \frac{1}{p(s) + t}f_{s}(s,t) \in \textrm{span} \{ f_{s}, f_{t} \}. \]
This leads to $h_{12}=0$. Thus, it holds that second fundamental form $h=0$. It is a plane by Proposition \ref{totgeod}. 
In summary, the case of (iii) is reduced to the case of (i), or gives a plane.
\subsection*{Case of (iv)} 
Since we can apply Proposition \ref{prop:3}, we have $C(s,t)=0$. Again, from (\ref{eq:2}) and (\ref{eq:3}), it follows that 
\[\epsilon\langle \gamma,x''\rangle_{p}\gamma=\gamma''t+x''.\]
When we compare the coefficients on the both sides with respect to $t$, we get
\begin{numcases}
  {}
  \gamma''(s) = 0, & \label{eq:8} \\
  x''(s)=\epsilon\langle \gamma(s),x''(s)\rangle_{p}\gamma(s). & \label{eq:9}
\end{numcases}
We get the following solution by referring to \cite[pp. 44--45]{An} again
\begin{align*}
f(s,t)=\bigg( t+ & \left. \frac{C_{1}}{2}s^2 \right )e_{1}+\sqrt{\frac{|C_{1}|}{2}} \left (\frac{C_{1}}{2}\frac{s^3}{3}+st-s \right )e_{2}\\
&\quad +\sqrt{\frac{|C_{1}|}{2}} \left (\frac{C_{1}}{2}\frac{s^3}{3}+st+s \right )e_{3}+\frac{\epsilon}{2C_{1}}\langle v, v \rangle_{p}e_{1}+x_{0},
\end{align*}
where $\{e_{1},e_{2},e_{3}\}$ is an O.N.S., and $v=x'(0),\ x_{0}=x(0)$. Moreover, $C_{1}$ is a non-zero constant number given by $C_{1} := - \epsilon \langle \gamma', x' \rangle_{p}$, where $\langle \gamma', x' \rangle_{p}$ is constant in this case. 
Finally, making a suitable translation and scaling, we may set $\frac{C_{1}}{2}=1, x_{0}=-\frac{\epsilon}{2C_{1}}\langle v, v \rangle_{p}e_{1}$. Thus, we obtain 
\begin{align*}
f(s,&t)=(t+s^2)e_{1} + \left (\frac{s^3}{3}+st-s \right )e_{2} + \left (\frac{s^3}{3}+st+s \right )e_{3}, \\
& \quad \textrm{where} \ \{e_{1},e_{2},e_{3}\} : \textrm{O.N.S.}, \ \langle e_{1}, e_{1} \rangle_{p}=\langle e_{2}, e_{2} \rangle_{p}=-\langle e_{3}, e_{3} \rangle_{p}=\pm1.
\end{align*}
\subsection*{Case of (v)}
Since we can apply Proposition \ref{prop:3}, we have $C(s,t)=0$. 
$C_{1}=- \epsilon \langle \gamma', x' \rangle_{p}=0$ implies that 
\begin{numcases}
  {}
  \gamma''(s) = 0, & \nonumber \\
  x''(s)=0. & \nonumber
\end{numcases}
These formulas give that 
\[\gamma(s)=se_{1}+e_{2}, \ \ x(s)=se_{3}+x_{0},\]
where $e_{1}=\gamma'(0), e_{2}=\gamma(0), e_{3}=x'(0), x_{0}=x(0) \in \mathbb{R}^{n}_{p}$. Thus, we have 
\[f(s,t)=\gamma(s)t+x(s)=ste_{1}+te_{2}+se_{3}+x_{0}.\]
In particular, from the condition of (v), we calculate that $\langle e_{1}, e_{1} \rangle_{p}=0$, $\langle e_{2}, e_{2} \rangle_{p}=\epsilon=\pm1, \langle e_{3}, e_{3} \rangle_{p}=\delta=\pm1, \langle e_{1},e_{2}\rangle_{p}=\langle \gamma(0),\gamma'(0)\rangle_{p}=0$. 
The facts that $g_{12}=0$, $\langle e_{2},e_{3}\rangle_{p}=0$ and $\frac{d}{ds}g_{12}=0$ imply that $\langle e_{3},e_{1}\rangle_{p}=0$. 
Finally, making a suitable translation, we may set $x_{0}=0$. Thus, we obtain 
\begin{align*}
f(s,&t)=ste_{1}+te_{2}+se_{3}, \\
& \textrm{where} \ \{e_{1},e_{2},e_{3}\} : \textrm{O.S.}, \ \langle e_{1}, e_{1} \rangle_{p}=0, \ \langle e_{2}, e_{2} \rangle_{p}=\pm1, \ \langle e_{3}, e_{3} \rangle_{p}=\pm1;
\end{align*}
where the double signs are arbitrary.
\subsection*{Case of (vi)}  
We will see $C(s,t)$ directly since we can apply neither Proposition \ref{prop:2} nor Proposition \ref{prop:3}. We have
\[C(s,t)=\frac{\frac{d}{ds}\langle \gamma',x'\rangle_{p}t}{2\langle \gamma',x'\rangle_{p}t}=\frac{\frac{d}{ds}\langle \gamma',x'\rangle_{p}}{2\langle \gamma',x'\rangle_{p}}.\]
Therefore, from (\ref{eq:2}) and (\ref{eq:3}), it holds
\begin{eqnarray}
\frac{\frac{d}{ds}\langle \gamma',x'\rangle_{p}}{2\langle \gamma',x'\rangle_{p}}(\gamma't+x')+\epsilon\langle \gamma,x''\rangle_{p}\gamma=\gamma''t+x'' \label{eq:11}.
\end{eqnarray}
We have not yet used the freedom that we enjoy in reparametrizing $\gamma(s)$ in the case (vi). By referring \cite[p.~44]{An}, we know that $C(u,t)=0$, where $u=u(s)$ is an another parameter.
So, $\langle \gamma',x'\rangle_{p}$ is  a non-zero constant, i.e. the reduction to the case in which the function $\langle \gamma',x'\rangle_{p}$ is a non-zero constant. If $\langle \gamma',x'\rangle_{p}$ is a non-zero constant, using formula (\ref{eq:11}) from $\frac{d}{ds}\langle \gamma',x'\rangle_{p}=0$, we get
\[\epsilon\langle \gamma,x''\rangle_{p}\gamma=\gamma''t+x''.\]
This means immediately the reduction to the case (iv). Hence, we don't get new results in this case.
Summing up the cases from (i) to (vi), the classification is completed by using Theorem \ref{thm:1} and Proposition \ref{prop:1}.
\begin{thm}\label{thm:2} Let $S$ be a non-planar, ruled minimal surface of pseudo-Euclidean space $\mathbb{R}^n_{p}$ in generic situations. Then, $S$ is locally homothetic to an open subset of the following surfaces.
\begin{enumerate}
\setlength{\leftskip}{-7.0mm}
  \item[1.] A minimal cylinder \\ 
\ $f(s,t)=\gamma_{0}t+x(s), \\ 
\ where \ \gamma_{0}:$ null vector, $x(s):$ null curve s.t. $\langle \gamma_{0},x'(s)\rangle_{p} \neq 0$;
  \item[2.] An elliptic helicoid of the first kind \\
\ $f(s,t)=(\cos{s}e_{1}+\sin{s}e_{2})t+se_{3}, \\ 
\ where \ \{e_{1},e_{2},e_{3}\}:O.N.S., \ \langle e_{1}, e_{1} \rangle_{p}=\langle e_{2}, e_{2} \rangle_{p}=\pm1, \ \langle e_{3}, e_{3} \rangle_{p}=\pm1$;
  \item[3.] An elliptic helicoid of the second kind \\
\ $f(s,t)=(\cos{s}e_{1}+\sin{s}e_{2})t+se_{3}, \\ 
\ where \ \{e_{1},e_{2},e_{3}\}:O.S., \ \langle e_{1}, e_{1} \rangle_{p}=\langle e_{2}, e_{2} \rangle_{p}=\pm1, \ \langle e_{3}, e_{3} \rangle_{p}=0$;
  \item[4.] A hyperbolic helicoid of the first kind \\
\ $f(s,t)=(\cosh{s}e_{1}+\sinh{s}e_{2})t+se_{3}, \\ 
\ where \ \{e_{1},e_{2},e_{3}\}:O.N.S., \ \langle e_{1}, e_{1} \rangle_{p}=-\langle e_{2}, e_{2} \rangle_{p}=\pm1, \ \langle e_{3}, e_{3} \rangle_{p}=\pm1$;
  \item[5.] A hyperbolic helicoid of the second kind \\
\ $f(s,t)=(\cosh{s}e_{1}+\sinh{s}e_{2})t+se_{3}, \\ 
\ where \ \{e_{1},e_{2},e_{3}\}:O.S., \ \langle e_{1}, e_{1} \rangle_{p}=-\langle e_{2}, e_{2} \rangle_{p}=\pm1, \ \langle e_{3}, e_{3} \rangle_{p}=0$;
  \item[6.] A parabolic helicoid \\
\ $f(s,t)=(t+s^2)e_{1}+\left(\dfrac{s^3}{3}+st-s\right)e_{2}+\left(\dfrac{s^3}{3}+st+s\right)e_{3}, \\ 
\ where \ \{e_{1},e_{2},e_{3}\}:O.N.S., \ \langle e_{1}, e_{1} \rangle_{p}=\langle e_{2}, e_{2} \rangle_{p}=-\langle e_{3}, e_{3} \rangle_{p}=\pm1$;
  \item[7.] A minimal hyperbolic paraboloid \\
\ $f(s,t)=ste_{1}+te_{2}+se_{3}, \\ 
\ where \ \{e_{1},e_{2},e_{3}\}:O.S., \ \langle e_{1}, e_{1} \rangle_{p}=0, \ \langle e_{2}, e_{2} \rangle_{p}=\pm1, \ \langle e_{3}, e_{3} \rangle_{p}=\pm1$;
\end{enumerate}
where the double signs are arbitrary.
\end{thm}
\begin{rem} \rm \label{rem:1}
Among surfaces of Theorem \ref{thm:2}, the surfaces of 3, 5 and 7 cannot be constructed in the case $n=3$. 
Actually, since the O.S. contains a null vector, we cannot take the two or more linearly independent vectors which orthogonal to it. 
But, it can be constructed when $n \geq 4, p \geq 1$. For instance, In the case $n=4, p=2$, if we consider
\[   
e_{1}
=
  \left(
    \begin{array}{c}
      1 \\
      0 \\
      0 \\
      0
    \end{array}
  \right)
, \ 
e_{2}
=
   \left(
    \begin{array}{c}
      0 \\
      0 \\
      1 \\
      0
    \end{array}
  \right)
, \ 
e_{3}
=
  \left(
    \begin{array}{c}
      0 \\
      1 \\
      0 \\
      1
    \end{array}
  \right)
,\]
then this is an example of the hyperbolic helicoids of the second kind. 
Moreover, from the consequence of this theorem, the non-planar ruled minimal surfaces of the canonical Euclidean space $(\mathbb{R}^n,\langle \cdotp,\cdotp \rangle_{0})$ is only the elliptic helicoid of the first kind, that is, the classical helicoid.
\end{rem} 
\begin{rem} \rm
We observe the non-degeneracy and the type change of each of surfaces (see also \cite[p.~46]{An}, \cite{Ca2}).\\
1. Regarding the minimal cylinder:
Since $\det g=-g_{12}^2=- \left (\langle \gamma_{0},x'\rangle_{p} \right )^2<0$, the surface is timelike.\\
2, 4. Regarding the elliptic or hyperbolic helicoid of the first kind:
In this case, we have $\det g= \left (\langle e_{2}, e_{2} \rangle_{p}t^2+\langle e_{3}, e_{3} \rangle_{p} \right ) \langle e_{1}, e_{1} \rangle_{p}$.
\begin{itemize}
\item When $\langle e_{1}, e_{1} \rangle_{p}=\langle e_{2}, e_{2} \rangle_{p}=\langle e_{3}, e_{3} \rangle_{p}$,\\
the corresponding surface is elliptic helicoid, and it is spacelike.
\item When $\langle e_{1}, e_{1} \rangle_{p}=\langle e_{2}, e_{2} \rangle_{p}=-\langle e_{3}, e_{3} \rangle_{p}$,\\
the corresponding surface is elliptic helicoid, and it is timelike for $t^2<1$ and spacelike for $t^2>1$.
\item When $\langle e_{1}, e_{1} \rangle_{p}=-\langle e_{2}, e_{2} \rangle_{p}=-\langle e_{3}, e_{3} \rangle_{p}$,\\
the corresponding surface is hyperbolic helicoid, and it is timelike.
\item When $\langle e_{1}, e_{1} \rangle_{p}=-\langle e_{2}, e_{2} \rangle_{p}=\langle e_{3}, e_{3} \rangle_{p}$,\\
the corresponding surface is hyperbolic helicoid, and it is spacelike for $t^2<1$ and timelike for $t^2>1$.
\end{itemize}
3, 5. Regarding the elliptic or hyperbolic helicoid of the second kind:
In this case, we have $\det g=\langle e_{1}, e_{1} \rangle_{p}\langle e_{2}, e_{2} \rangle_{p}t^2$.
\begin{itemize}
\item When $\langle e_{1}, e_{1} \rangle_{p}=\langle e_{2}, e_{2} \rangle_{p}$,\\
the corresponding surface is elliptic helicoid, and it is spacelike for $t \neq 0$.
\item When $\langle e_{1}, e_{1} \rangle_{p}=-\langle e_{2}, e_{2} \rangle_{p}$,\\
the corresponding surface is hyperbolic helicoid, and it is timelike for $t \neq 0$.
\end{itemize}
6. Regarding the parabolic helicoid:
In this case, we have $\det g=-4\langle e_{1}, e_{1} \rangle_{p}t$.
\begin{itemize}
\item When $\langle e_{1}, e_{1} \rangle_{p}=1$,\\
it is timelike for $t>0$ and spacelike for $t<0$.
\item When $\langle e_{1}, e_{1} \rangle_{p}=-1$,\\
it is spacelike for $t>0$ and timelike for $t<0$.
\end{itemize}
7. Regarding the minimal hyperbolic paraboloid:
In this case, we have $\det g=\langle e_{1}, e_{1} \rangle_{p}\langle e_{2}, e_{2} \rangle_{p}$.\\
It is spacelike when $\langle e_{1}, e_{1} \rangle_{p}=\langle e_{2}, e_{2} \rangle_{p}$, and it is timelike when $\langle e_{1}, e_{1} \rangle_{p}=-\langle e_{2}, e_{2} \rangle_{p}$.
\end{rem}
\begin{thm}\label{thm:3} 
There is no hyperbolic helicoid of the second kind in Minkowski $n$-space $\mathbb{R}^{n}_{1}$. And, there is no elliptic helicoid of the second kind in four dimensional pseudo-Euclidean space $\mathbb{R}^{4}_{2}$ with the neutral metric.
\end{thm}
\begin{proof} 
The group of linear isometries of $\mathbb{R}^n_{p}$ preserving the orientation coincides with the pseudo-rotation group $SO(p,n-p)$, which is a subgroup of pseudo-orthogonal group $O(p,n-p)$ (refer to \cite{O}). 

Regarding the former, we prove the non-existence of the hyperbolic helicoid of the second kind by contradiction : we assume that there is a hyperbolic helicoid of the second kind.
Let $\{e_{1}, e_{2}, e_{3}\}$ be an O.S. on $\mathbb{R}^n_{1}$, which realizes a hyperbolic helicoid of the second kind. Using the transitivity of $SO(p,n-p)$, we may set $e_{1}=(1,0, \cdots ,0)$. Since $\{e_{1}, e_{2}, e_{3}\}$ are orthogonal to each other, we can put $e_{2}=(0,a_{2}, \cdots ,a_{n}), \ e_{3}=(0,b_{2}, \cdots ,b_{n})$. Since $\langle e_{3}, e_{3} \rangle_{1}=0$ implies $b_{2}^2+ \cdots +b_{n}^2=0$, we obtain $e_{3}=0$. But this contradicts that $e_{3}$ is a null vector. 

Regarding the latter, from Remark \ref{rem:1}, we can see the existence of hyperbolic helicoid of the second kind. Similarly, we prove the non-existence of the elliptic helicoid of the second kind by contradiction. 
Let $\{e_{1}, e_{2}, e_{3}\}$ be an O.S. on $\mathbb{R}^4_{2}$, which realizes an elliptic helicoid of the second kind. In case that $\langle e_{1}, e_{1} \rangle_{2} = \langle e_{2}, e_{2} \rangle_{2} = -1$, we may assume that $e_{1}=(1,0,0,0)$. Then we can express $e_{2}=(0,a,b,c), \ e_{3}=(0,x,y,z)$. The fact that $\{e_{1}, e_{2}, e_{3}\}$ is an O.S. gives equations
\begin{numcases}
  {}
  -a^2+b^2+c^2 = -1, & \label{eq:12} \\ 
  -x^2+y^2+z^2 = 0, & \label{eq:13} \\ 
  -ax+by+cz = 0. & \label{eq:14}
\end{numcases}
Since we have $x \neq 0$ from (\ref{eq:13}), $a=\dfrac{by+cz}{x}$ holds by (\ref{eq:14}). Substituting the formula (\ref{eq:12}) for this, 
\begin{eqnarray*}
-\left(\dfrac{by+cz}{x}\right)^2+b^2+c^2&=&\dfrac{x^2-y^2}{x^2}b^2-\dfrac{2yz}{x^2}bc+\dfrac{x^2-z^2}{x^2}c^2 \\
&=&\dfrac{z^2}{x^2}b^2-\dfrac{2yz}{x^2}bc+\dfrac{y^2}{x^2}c^2 \\
&=&\dfrac{b^{2}z^2-2bczy+c^{2}y^2}{x^2}=\left(\dfrac{bz-cy}{x}\right)^2=-1.
\end{eqnarray*}
This is a contradiction. In case that $\langle e_{1}, e_{1} \rangle_{2} = \langle e_{2}, e_{2} \rangle_{2} = 1$, we may assume that $e_{1} = (0, 0, 0, 1)$, and we can discuss as above.
\end{proof}
\begin{rem} \label{rem10} \rm
We can summarize the existence by the table indicated below.
\begin{table}[!h]
\begin{tabular}{|c|c|c|c|c|c|c|c|}
\hline
$\empty$ & 1 & 2 & 3 & 4 & 5 & 6 & 7 \\ \hline
$\mathbb{R}^n_{0}(n\geq3)$ & $\times$ & $\bigcirc$ & $\times$ & $\times$ & $\times$ & $\times$ & $\times$ \\ \hline
$\mathbb{R}^3_{1}$ & $\bigcirc$ & $\bigcirc$ & $\times$ & $\bigcirc$ & $\times$ & $\bigcirc$ & $\times$ \\ \hline
$\mathbb{R}^4_{1}$ & $\bigcirc$ & $\bigcirc$ & $\bigcirc$ & $\bigcirc$ & $\times$ & $\bigcirc$ & $\bigcirc$ \\ \hline
$\mathbb{R}^4_{2}$ & $\bigcirc$ & $\bigcirc$ & $\times$ & $\bigcirc$ & $\bigcirc$ & $\bigcirc$ & $\bigcirc$ \\ \hline
$\mathbb{R}^n_{1}(n\geq5)$ & $\bigcirc$ & $\bigcirc$ & $\bigcirc$ & $\bigcirc$ & $\times$ & $\bigcirc$ & $\bigcirc$ \\ \hline
$\mathbb{R}^n_{p}(n\geq5,2\leq p\leq q)$ & $\bigcirc$ & $\bigcirc$ & $\bigcirc$ & $\bigcirc$ & $\bigcirc$ & $\bigcirc$ & $\bigcirc$ \\ \hline
\end{tabular} 
\caption{} \label{tab:2}
\end{table}
Here, the numbers 1, \ldots, 7 of this table correspond to that of Theorem \ref{thm:2}, and the symbols $\bigcirc$ and $\times$ express the existence and non-existence respectively.  
Actually, since we have the explicit expressions for the surfaces 1, \ldots, 7 from Theorem \ref{thm:2}, we can check the other existence by taking suitable vectors.
\end{rem}
\begin{rem} \label{rem11} \rm
The following surfaces are embedded in three dimensional subspace with degenerate metric of pseudo-Euclidean space:
    \begin{itemize}
 \item an elliptic helicoid of the second kind,
 \item a hyperbolic helicoid of the second kind,
 \item a minimal hyperbolic paraboloid. 
    \end{itemize}
We call a manifold $M$ \textit{singular semi-Riemannian} if $M$ is endowed with a degenerate metric  (\cite{St}). O. C. Stoica mentions that singular semi-Riemannian manifolds relate with General Relativity remarkably. The research of submanifolds in singular semi-Riemannian manifolds is a few. One of the motivations to research these objects is the following results.
\begin{enumerate}
\item[(A)] Minimal entire graphs in $\mathbb{R}^3_{0}$ are planes only (\cite{Be}).
\item[(B)] Spacelike minimal entire graphs in $\mathbb{R}^3_{1}$ are also planes only (\cite{Ca1}).
\item[(C)] There exists non-trivial examples of timelike minimal entire graphs in $\mathbb{R}^3_{1}$ (\cite{FKKRUY}).
\end{enumerate}
Classically, these results are also known as problems of Bernstein type. The study of their generalizations still continues. The minimal hyperbolic paraboloid is a spacelike minimal surface when $\langle e_{2}, e_{2} \rangle_{p}=\langle e_{3}, e_{3} \rangle_{p}=\pm 1$. This gives a non-trivial example of minimal entire graphs. In case of four dimensional Minkowski space, A. Honda and S. Izumiya \cite{HI} proved these results as the general form by another point of view.
\end{rem}
\section*{Acknowledgments}
The author would like to express his deepest gratitude to his advisor, Professor
Takashi Sakai. The author is also very grateful to advisor's laboratory members, Shuhei Kano and Yasunori Terauchi for their comments and valuable advices.


\begin{thebibliography}{9}
\bibitem{An} 
\newblock {H. Anciaux},
\newblock \emph{
Minimal submanifolds in pseudo-Riemannian geometry}, 
\newblock  {World Scientific} (2011).
%
\bibitem{Be}
\newblock {S. Bernstein}, 
\newblock \emph{
Sur une th$\acute{\textrm{e}}$or$\grave{\textrm{e}}$me de g$\acute{\textrm{e}}$ometrie et ses applications aux $\acute{\textrm{e}}$quations d$\acute{\textrm{e}}$riv$\acute{\textrm{e}}$es partielles du type elliptique}, 
\newblock {Comm. Soc. Math. Kharkov.}, \textbf{15}, (1914), 38--45.
%
\bibitem{Ca1} 
\newblock {E. Calabi}, 
\newblock \emph{
Examples of Bernstein problems for some nonlinear equations}, 
\newblock  {Proc. Symp. Pure Math.}, \textbf{15}, (1970), 223--230.
%
\bibitem{Ca2} 
\newblock {E. Catalan}, 
\newblock \emph{
Sur les surfaces r$\acute{\textrm{e}}$gl$\acute{\textrm{e}}$es dont l'aire est un minimum}, 
\newblock  {J. Math. Pure Appl.}, \textbf{7}, (1842), 203--211.
%
\bibitem{FKKRUY} 
\newblock {S. Fujimori, Y. Kawakami, M. Kokubu, W. Rossman, M. Umehara and K. Yamada}, 
\newblock \emph{Entire zero-mean
curvature graphs of mixed type in Lorentz-Minkowski 3-space}, 
\newblock {Quarterly J. Math.}, \textbf{67}, (2016), 801--837.
%
\bibitem{HI} 
\newblock {A. Honda, S. Izumiya}, 
\newblock \emph{The lightlike geometry of marginally trapped surfaces in Minkowski
space-time}, 
\newblock {J. Geom.}, \textbf{106}, (2015), 185--210.
%
\bibitem{KY1} 
\newblock {Y. H. Kim, D. W. Yoon}, 
\newblock \emph{On the Gauss map of ruled surfaces in Minkowski space}, 
\newblock{Rocky Mountain
J. Math.}, \textbf{35}, (2005), no. 5, 1555--1581.
%
\bibitem{OK} 
\newblock {O. Kobayashi}, 
\newblock \emph{Maximal surfaces in the 3-dimensional Minkowski space $L^3$}, 
\newblock{Tokyo J. Math.}, \textbf{6}, (1983), no. 2, 297--309.
%
\bibitem{O} 
\newblock {M. O'Neill}, 
\newblock \emph{Semi-Riemannian geometry with applications to relativity}, 
\newblock{Academic Press, London} (1983).
%
\bibitem{St} 
\newblock {O. C. Stoica}, 
\newblock \emph{On singular semi-Riemannian manifolds}, 
\newblock {Int. J. Geom. Methods Mod. Phys.}, \textbf{11}, (2014), no. 5, 1450041.
%
\bibitem{Wo} 
\newblock {I. van de Woestijne}, 
\newblock \emph{Minimal surfaces of the 3-dimensional Minkowski space}, 
\newblock {in Geometry and Topology of submanifolds II., M. Boyom, J.-M. Morvan and L.Verstraelen Eds, World Scientific} (1990), 344--369.
\end{thebibliography}
\end{document}